\documentclass[notitlepage,leqno,11pt]{article}
\usepackage{amssymb}
\catcode`\@=11 \@addtoreset{equation}{section}

\catcode`\@=12

\usepackage{xcolor}
\usepackage{latexsym}
\usepackage{textcomp}
\usepackage{amsmath}
\usepackage{amsfonts}
\usepackage{amssymb}
\usepackage{mathrsfs}
\usepackage[numbers]{natbib}
\usepackage{hyperref}
\usepackage{enumerate}

\hypersetup{
colorlinks=true,
linkcolor=black,
citecolor=red,
filecolor=black,
urlcolor=black,
}
\newtheorem{Theorem}{Theorem}[section]
\newtheorem{theorem}{Theorem}[section]

\newtheorem{definition}[Theorem]{Definition}

\newtheorem{lemma}[Theorem]{Lemma}

\newtheorem{proposition}[Theorem]{Proposition}
\newtheorem{remark}[Theorem]{Remark}

\numberwithin{equation}{section}

\def \dis {\displaystyle}

\newcommand{\bOm}{\overline{\Omega}}

\newcommand{\be}{\begin{equation}}
\newcommand{\ee}{\end{equation}}
\newcommand{\bdis}{\begin{displaymath}}
\newcommand{\edis}{\end{displaymath}}

\newcommand{\R}{\mathbb{R}}
\newcommand{\Rn}{\mathbb{R}^N}

\def \O {\overline{\Omega} }

\newcommand{\NN}{\mathcal{N}_{s}}
\newcommand{\N}{\mathbb{N}}

\def \R {\mathbb{R}}
\def \N {\mathbb{N}}

\DeclareSymbolFont{pxfontssymbolsC}{U}{pxsyc}{m}{n}
\DeclareMathSymbol{\coloneqq}{\mathrel}{pxfontssymbolsC}{66}

\title{Optimal control of diffusion equation with missing data governed by Dirichlet fractional Laplacian}
\author{J-D.~Djida, P.F.~Soh and G.~Mophou}

\begin{document}
\date{}
\maketitle
\let\thefootnote\relax\footnotetext{jeandaniel.djida@usc.es (J-D. Djida), pasquini.soh@aims-cameroon.org (P.F. Soh), gisele.mophou@univ-antilles.fr (G. Mophou ).}

\let\thefootnote\relax\footnotetext{Departamento de Estat\'{\i}stica, An{\'a}lise Matem\'{a}tica e Optimizaci\'on, Universidade de Santiago de Compostela, 15782 Santiago de Compostela, Spain and African Institute for Mathematical Sciences (AIMS), P.O. Box 608, Limbe Crystal Gardens, South West Region, Cameroon}
\let\thefootnote\relax\footnotetext{African Institute for Mathematical Sciences (AIMS), P.O. Box 608, Limbe Crystal Gardens, South West Region, Cameroon}
\let\thefootnote\relax\footnotetext{African Institute for Mathematical Sciences (AIMS), P.O. Box 608, Limbe Crystal Gardens, South West Region, Cameroon and Laboratoire L.A.M.I.A., D\'{e}partement de Math\'{e}matiques et Informatique, Universit\'{e} des Antilles, Campus Fouillole, 97159 Pointe-\`{a}-Pitre,(FWI), Guadeloupe.}
\bigskip

\begin{abstract}
We consider an optimal control problem of diffusion equation with missing data governed by the fractional Laplacian with homogeneous Dirichlet boundary conditions on an arbitrary interaction domain disjoint from the domain of the state equation. We assume that the unknown initial condition belongs to an appropriate space of infinite dimension, the so-called space of uncertainties. The key tools we used in order to characterize the optimal control is the no-regret and low-regret control developed by J.L Lions. 
\end{abstract}

\section{Introduction}\label{sec:introduction} 
The motivation for the growing interest in studying evolution equations involving fractional Laplace operator relies in the large number of possible applications in the modeling of several complex natural phenomena for which a local approach turns up to be inappropriate or limiting such as anomalous transport and diffusion \cite{Bologna, Meerschaert}. Indeed, there is  ample situations in which a nonlocal equation gives a significantly better description of a phenomenon than a PDE. Such models with nonlocal operators are now experiencing impressive applications in different subjects, among others, we mention applications in optimization~\cite{DL76}, finance~\cite{CT04}, the thin obstacle problem~\cite{Sil07, MS08}, porous media flow \cite{JLVAZ}, continuum mechanics, population dynamics, stochastic processes of L\'{e}vy type, phase transitions~\cite{CSM05, SiV09, FV11}, stratified materials~\cite{SV09, CV09, CV10}, anomalous diffusion~\cite{MK00, JLVAZ, MMM11}, crystal dislocation~\cite{GM10, BKM10}, see also~\cite{Sil05,Sil07,fall-schroding-17} for further motivation and applications.
The controlled  fractional diffusion equation we consider in this paper is of missing initial data.  Such models could describe diffusion of pollution in a porous media. We assume that we do not know when the phenomenon began. To solve this problem we use the notion of no-regret and least regret control \cite{Lions1992}. There are few works in the literature using these concepts of optimal control \cite{Nakoulima2000, Nakoulima2002, Jacob2010, Nakoulima2004, Gabay, Lions2000, Nakoulima2003, Nakoulima1999, Lions1999}. In those papers one can observe that this concept of no regret and low-regret find their application on the control of distributed linear systems possessing with missing data. A generalization of this approach for some nonlinear distributed systems possessing incomplete data has been also developed in the reference \cite{Nakoulima2002}. We stress out that this concept found also its usefulness in the control of population dynamics, propagation of pollution problems \cite{Nakoulima2000, Jacob2010}. Recently in \cite{Mophou2015} the author use this concept to control fractional diffusion equation with missing boundary condition. The authors in \cite{baleanu} applied the same concept to a fractional wave equation with missing initial velocity of state. As far as we know this concept of optimal control has not been  used to study evolution equation with missing initial data involving fractional Laplacian operators. So, in this paper we first prove that there exists a unique low-regret control which can bring  the state of the considered fractional diffusion equation to a desired state. As the low regret control is obtained by relaxing the cost associated to the no-regret control, we obtain that the low-regret control converge towards the no-regret control that we characterize with an optimality system. \par 

The rest of paper is organized as follows. In Section \ref{Sect2}, we give some notations and definitions of functional spaces and their associated norms for the need of this work. We also recall some results on existence and uniqueness of the weak solution of the considered nonlocal fractional diffusion equation. In Section \ref{Sec:FormPB}, we give the formulation of the problem that will be analyzed through this paper and recall the main concept of no and low regret control. Finally in Section \ref{Sec:No-low-concept}, we study the low regret and the no regret control problem and give the optimality system that characterizes each control. 

\section{Preliminaries}\label{Sect2}
In this section, we start by introducing some spaces and their norms which will be used throughout the paper. Then we provide some elementary properties of the fractional Laplacian $(-\Delta)^{s}$ that we will need. \medskip

Let $\Omega$ be an open bounded subset of $\R^N, \, N\in \N\setminus \{0\}$  with Lipschitz boundary. For any $s\in (0,1)$ and $ p\in [1,+\infty)$, we recall that the fractional Sobolev space $W^{s,p}(\Omega )$ is defined as follows:
\[
W^{s,p}(\Omega ):=\left\{ w\in L^{p}(\Omega):\;\int_{\Omega\times \Omega}
\frac{|w(x)-w(y)|^{p}}{|x-y|^{N+ps}}dx~dy<\infty \right\}.
\]
It is endowed with the natural norm
\[
\Vert w\Vert _{W^{s,p}(\Omega )}:=\left( \int_{\Omega }|w|^{p}\;dx+\int_{\Omega\times\Omega }\frac{|w(x)-w(y)|^{p}}{|x-y|^{N+ps}}
dx~dy\right) ^{\frac{1}{p}}.
\]
we set
\[
W_0^{s,2}(\O):= \left\lbrace w\in W^{s,2}(\R^N):\; w=0\;\mbox{ on }\;\R^N\setminus\Omega\right\rbrace.
\]
If we denote by $\mathbb{D}(\Omega)$ the space of  continuously infinitely differentiable functions with compact support in $\Omega$, then $\mathbb{D}(\Omega)$ is dense in  $W_0^{s,2}(\O)$ because the boundary of $\Omega$ is Lipschitz and continuous ( see \cite{fice}).
We then  denote by $W^{-s,2}(\O)= \left(W^{s,2}_{0}(\O)\right)^{\prime}$, the dual of the Hilbert space $W^{s,p}_{0}(\O)$.\medskip

We recall the following continuous embedding of fractional Sobolev spaces.
 \begin{enumerate}[1)]
 \item if $0 < s \leq s' <1$,
\[
W^{s',2}(\O) \hookrightarrow W^{s,2}(\O).
\]
 Hence there exists a constant $C(N,s) \geq 1$ such that, for any $w \in W^{s',2}(\Omega)$ :
\[
\left\|w\right\|_{W^{s,2}(\O)} \leq C(N,s) \left\|w\right\|_{W^{s',2}(\O)}.
\]
\item for $0<s<1$,
\[
W_0^{s,2}(\O)\hookrightarrow L^2(\Omega)\hookrightarrow W^{-s,2}(\O).
\]
\end{enumerate}
\begin{remark}\label{remcont}
Set
\begin{equation}\label{defW}
W(0,T):=\dis \left\{\rho \, \big| \, \,\rho \in L^{2}\left((0,T);W^{s,2}_{0}(\O)\right), \dis \partial_{t}\rho \in L^{2}\left((0,T);W^{-s,2}(\O)\right)\right\}.
\end{equation}
Then from \cite[Theorem 10]{Tommaso}, we have $ W(0,T)\subset \mathcal{C}\left([0,T];L^{2}(\Omega)\right)$
\end{remark}

Next we recall the definition of the fractional nonlocal operators that we are interested in this work, the fractional laplacian in the integral formulation.\medskip

Let us consider the following space
\[
\mathcal{L}_{s}^{1}(\R^N):=\left\{w:\R^N\to\R\;\mbox{ measurable},\; \int_{\R^N}\frac{|w(x)|}{(1+|x|)^{N+2s}}\;dx<\infty\right\}.
\]
For $w\in \mathcal L_s^{1}(\R^N)\cap \mathcal{C}^{2}_{loc}(\R^N)$ and for a positive small enough constant $\varepsilon$ we set
\[
(-\Delta)_\varepsilon^s w(x):= C_{N,s}\int_{\{y\in\R^N:\;|x-y|>\varepsilon\}}\frac{w(x)-w(y)}{|x-y|^{N+2s}}\;dy,\;\;x\in\R^N,
\]
where $C_{N,s}$ is a normalization constant, given by
\[
C_{N,s}:=\frac{s2^{2s}\Gamma\left(\frac{2s+N}{2}\right)}{\pi^{\frac
N2}\Gamma(1-s)}.
\]

The {\bf fractional Laplacian} $(-\Delta)^s$ of $w \in \mathcal L_s^{1}(\R^N)\cap \mathcal{C}^{2}_{loc}(\R^N)$ is defined by the following singular integral:
\begin{equation}\label{Eq:fractional_Laplacian_def}
(-\Delta)^{s}w(x)=C_{N,s}\,\mbox{P.V.}\int_{\R^N}\frac{w(x)-w(y)}{|x-y|^{N+2s}}\;dy=\lim_{\varepsilon\downarrow 0}(-\Delta)_\varepsilon^s w(x),\;\;x\in\R^N,
\end{equation}
provided that the limit exists.\medskip

For more details on the fractional Laplace operator we refer to \cite{BBC,Caf3,NPV,GW-CPDE,War,War-In} and references therein.\medskip

For $w\in W^{s,2}(\R^N)$ we recall the {\bf nonlocal normal derivative $\NN$} given by
\begin{equation}\label{Eq: NLND}
\NN w(x):=C_{N,s}\int_{\Omega}\frac{w(x)-w(y)}{|x-y|^{N+2s}}\;dy,\;\;\;x\in\R^N\setminus\bOm.
\end{equation}
The following result is taken from \cite[Lemma 3.2]{GSU} and \cite[Lemma 3.6]{Warma2018}.

\begin{lemma}\label{lem:GSU}
The operator  $\NN$ maps $W^{s,2}(\R^N)$ into $W_{\rm loc}^{s,2}(\R^{N}\setminus\Omega)$, where
\[
W_{\rm loc}^{s,2}(\Omega) := \left\lbrace w \in L^{2}(\Omega):~ w \varphi \in W^{s,2}(\Omega), \quad \forall ~~ \varphi \in \mathbb{D}(\Omega) \right\rbrace
\]
\end{lemma}
It follows from Lemma \ref{lem:GSU}, that if $w \in W^{s,2}(\R^N)$, then $\NN w\in L^2(\R^{N}\setminus\Omega)$.\medskip

We thus have  the integration by parts formula.
\begin{proposition}\label{propintegration_parts}
Let $w\in W^{s,2}(\R^N)$ be such that $(-\Delta)^{s} w \in L^2(\Omega)$. Then for every $v\in W^{s,2}(\R^N)$, the identity
\begin{equation}\label{Eq:Int-Part}
\frac{C_{N,s}}{2}\int\int_{\R^{2N}\setminus \left(\R^{N}\setminus \Omega \right)^2}\frac{\left(w(x)-w(y)\right)\left(v(x)-v(y)\right)}{|x-y|^{N+2s}}~dx~dy
=\int_{\Omega}v(-\Delta)^{s}w~dx + \int_{\R^{N}\setminus \Omega} v \NN w~dx,
\end{equation}
holds, where
\[
\R^{2N}\setminus(\R^{N}\setminus \Omega)^2=(\Omega\times\Omega)\cup(\Omega\times(\R^{N}\setminus \Omega))
\cup((\R^{N}\setminus \Omega)\times\Omega).
\]
Moreover, if $w=0$ in $(\R^{N}\setminus\Omega)$ then \eqref{Eq:Int-Part} can be rewritten as
\begin{equation}\label{Eq:Int-Partbis}
\frac{C_{N,s}}{2}\int_{\R^N}\int_{\R^{N}}\frac{\left(w(x)-w(y)\right)\left(v(x)-v(y)\right)}{|x-y|^{N+2s}}~dx~dy
=\int_{\Omega}v(-\Delta)^{s}w\;dx+\int_{\R^{N}\setminus \Omega}v \NN w~dx.
\end{equation}
\end{proposition}

Consider the following nonlocal fractional diffusion equation:
\begin{equation}\label{eqintro}
\left\lbrace
\begin{array}{lll}
\dis \partial_{t} p + (-\Delta)^{s}p &= f  & \textrm{ in } Q, \\
p &= 0 & \textrm{ on } \Xi,\\
p(0,x) &= p^0 & \textrm{ in } \Omega,
\end{array}
\right.
\end{equation}
where $ 0<s<1,$ $f\in L^2(Q)$, $p^0\in L^2(\Omega)$ and $\Xi:= (\R^{N}\setminus\Omega)\times (0,T)$, for $T>0$. The associated definition of weak solutions of \eqref{eqintro} with finite energy is given as follows.
\begin{definition}\label{weak_sol_def_fe}
We say that $p\in W(0,T)$ is a finite energy solution to problem \eqref{eqintro}, if the identity
\begin{equation}\label{weak-sol-fe}
\begin{array}{llllll}
\dis \int_{0}^T\int_\Omega w~\partial_{t}p~ dx~dt + \dis \frac{C_{N,s}}{2}\int_0^T\int_{\R^N}
\int_{\R^N}\frac{(p(x,t)-p(y,t))(w(x,t)-w(y,t))}{|x-y|^{N+2s}}\,dx~dy~dt =\\
\dis  \int_{0}^T \int_\Omega fv~dx~dt,
\end{array}
\end{equation}
holds, for any $w\in L^2((0,T);W_0^{s,2}(\O))$.
\end{definition}
Thanks to the continuous imbedding of $L^2(Q)$ into $ L^2((0,T);W^{-s,2}(\O))$, we have the following results\cite{Tommaso}.
\begin{proposition}\label{prop1}
Assume that $f\in L^2(Q)$, then for any $p^0\in L^2(\Omega)$ problem \eqref{eqintro} has a unique finite energy solution  in $W(0,T)$.\medskip
\end{proposition}
\begin{remark}\label{rem1} Since  $\dis \partial_{t}p\in L^2((0,T);W^{-s,2}(\O))$, the first integral in \eqref{weak-sol-fe} should be written
\[
\int_0^T
\left\langle \partial_{t}p , w \right\rangle_{W^{-s,2}(\O),W^{s,2}_0(\O)} dt.
\]
But using an appropriate change of variable and a cut-off argument, we can prove as in \cite{BWZ2} that if $p$ is a unique finite energy solution of \eqref{eqintro} then $\dis \partial_{t} p \in L^2(Q)$. Moreover we have the following estimate
\begin{eqnarray}
\dis \left\|\partial_{t}p\right\|_{L^2(Q)}&\leq & C(s,N,T)\left(\left\|f\right\|_{L^2(Q)}+\left\|p^0 \right\|_{L^2(\Omega)}\right)\label{estpt}\\
\dis \left\|p\right\|_{W^{s,2}_0(\O)}&\leq & C(s,N,T)\left(\left\|f\right\|_{L^2(Q)}+\left\|p^0 \right\|_{L^2(\Omega)}\right),\label{estp}
\end{eqnarray}
where $C(s,N,T)>0$ is a positive constant depending on $s,N$ and $T$.
\end{remark}

\begin{remark}\label{cont_unique}
From the Remark \ref{rem1}, we have that $p$ is solution of
\begin{equation}\label{eqintro1}
\left\lbrace
\begin{array}{lll}
\dis \partial_{t}p + (-\Delta)^{s}p &= 0  & \textrm{ in } Q, \\
p &= 0 & \textrm{ on } \Xi,\\
p(0,x) &= 0 & \textrm{ in } \Omega,
\end{array}
\right.
\end{equation}
with $ 0<s<1$ is identically zero.
\end{remark}
\begin{remark}
By a change of variable $t \mapsto T-t$, equation \eqref{eqintro} can be rewritten as
\begin{equation}\label{eqintro2}
\left\lbrace
\begin{array}{lll}
-\dis \partial_{t}\tilde{p} + (-\Delta)^{s}\tilde{p} &= \tilde{f}  & \textrm{ in } Q, \\
\tilde{p} &= 0 & \textrm{ on } \Xi,\\
\tilde{p}(T,x) &= p^0 & \textrm{ in } \Omega,
\end{array}
\right.
\end{equation}
where $ 0<s<1,$ $\tilde{p}= p(x,T-t),$ $\tilde{f}= f(T-t)$.  Therefore, $\tilde{p}$ is a finite energy solution to  problem \eqref{eqintro2}, if the identity
\begin{equation}\label{weak-sol-fe1}
\begin{array}{llllll}
-\dis \int_{0}^T\int_\Omega w \partial_{t}\tilde{p} dx~dt + \dis \frac{C_{N,s}}{2}\int_0^T\int_{\R^N}
\int_{\R^N}\frac{(\tilde{p}(x,t)-\tilde{p}(y,t))(w(x,t)-w(y,t))}{|x-y|^{N+2s}}\,dx~dy~dt =\\
\dis  \int_{0}^T \int_\Omega \tilde{f}v~dx~dt,
	\end{array}
\end{equation}
holds, for any $w\in L^2((0,T);W_0^{s,2}(\O))$. Moreover, if $\tilde{f}\in L^2(Q)$, then for any $p^0\in L^2(\Omega)$ problem \eqref{eqintro2} has a unique finite energy solution in $W(0,T)$ and the following estimates holds:
 \begin{eqnarray}
\dis \left\|\partial_{t} \tilde{p}\right\|_{L^2(Q)}&\leq & C(s,N,T)\left(\left\|\tilde{f}\right\|_{L^2(Q)}+\left\|p^0 \right\|_{L^2(\Omega)}\right)\label{estptb}\\
\dis \left\|\tilde{p}\right\|_{W^{s,2}_0(\O)}&\leq & C(s,N,T)\left(\left\|\tilde{f}\right\|_{L^2(Q)}+\left\|p^0 \right\|_{L^2(\Omega)}\right),\label{estpb}
\end{eqnarray}
where $C(s,N,T)>0$ is a positive constant depending on $s,N$ and $T$.
\end{remark}

\section{Formulation of the problem}\label{Sec:FormPB}
Let $s\in(0,1)$ and let also $\Omega$ be an open bounded subset of $\R^N, \, N\in \N\setminus \{0\}$  with Lipschitz boundary.
For any  time $T>0$ we set $Q = \R^N \times (0,T)$,  $\Xi=(\Rn \setminus \Omega) \times (0,T) $  and we consider the following nonlocal fractional diffusion equation:
\begin{equation}\label{Eq:main}
\left\lbrace
\begin{array}{lll}
\partial_{t}q + (-\Delta)^{s}q &= f + v & \textrm{ in } Q, \\
q &= 0 & \textrm{ on } \Xi,\\
q(0,x) &= g & \textrm{ in } \Omega,
\end{array}
\right.
\end{equation}
where $f$ in $ L^{2}(Q),$ the control function $v$ in $L^{2}(Q)$. The function $g\in L^2(\Omega)$ is unknown. Under the above assumption on the data, it has been shown that from Proposition \ref{prop1} the problem \eqref{Eq:main} admits a solution $q(v,g) =q(x,t;v,g)\in W(0,T)$ that depends on the control $v$ and on the missing initial data $g$. Since we want to bring the state $q$ solution of \eqref{Eq:main} to a desired state $z_d\in L^2(Q)$, we consider the cost function
$$
J(v,g) = \left \|q(v,g)-z_{d}\right \|_{L^{2}(Q)}^{2} + \aleph \left \| v \right \|^{2}_{L^{2}(Q)},
$$
where $\aleph > 0$. Observing that  the optimal control problem
$$
\inf_{v\in \mathscr{U}_{ad}}J(v,g), \qquad \forall~g\in L^2(Q),
$$
 has no sense, we are interested for any $\gamma>0$ in the following inf-sup problem:
\begin{equation}\label{Eq: no-regret}
\inf_{v\in \mathscr{U}_{ad}}\sup_{g\in L^2(\Omega)}\left[J_\gamma(v,g)-J_\gamma(0,g)\right],
\end{equation}

where 
\begin{equation}\label{Eq:functional}
J_\gamma(v,g) = \left \|q(v,g)-z_{d}\right \|_{L^{2}(Q)}^{2} + \aleph \left \| v \right \|^{2}_{L^{2}(Q)}-\gamma \|g\|^2_{L^2(\Omega)}.
\end{equation}
This problem is called low-regret control problem. It was introduced by J. L. Lions \cite{Lions1992}. This concept is an extension of the no-regret control also introduced by Lions which is stated as follows: find $v\in L^2(Q)$ solution of
\begin{equation}\label{Eq: no-regretintro}
\inf_{v\in \mathscr{U}_{ad}}\sup_{g\in L^2(\Omega)}\left[J(v,g)-J(0,g)\right].
\end{equation}
Actually, the control we are looking for is the no-regret control. But this  control problem is in general  difficult to characterize. That is the reason why we start by  studying the  low-regret problem \eqref{Eq: no-regret} which is easy to characterize. Then  we will show  that  when $\gamma$ tends zero this control  tends to the no-regret control. Finally we will obtain an optimality system for the no-regret control as a limit of the optimality system of the low-regret control.
	
\section{Study of inf-sup problem }\label{Sec:No-low-concept}

The aim of this section is to solve inf-sup problem given by the problem \eqref{Eq: no-regret}. Let $q(v,g):= q(x,t;v,g) \in W(0,T)$ be solution of \eqref{Eq:main}. Consider $q(v,0):=q(x,t;v,0)$, $q(0,g):=q(x,t;0,g)$ and  $q(0,0):=q(x,t;0,0)$ respectively solutions of
\begin{equation}\label{Eq:decomp1}
\left\lbrace
\begin{array}{lll}
\partial_{t}q(v,0) + (-\Delta)^{s}q(v,0) &= f + v &\textrm{ in } Q, \\
q(v,0) &= 0 &\textrm{ on } \Xi,\\
q(0;v,0) &= 0 & \textrm{ in } \Omega,
\end{array}
\right.
\end{equation}
\begin{equation}\label{Eq:decomp2}
\left\lbrace
\begin{array}{lll}
\partial_{t}q(0,g) + (-\Delta)^{s}q(0,g) &= f \textcolor{white}{+ v} &\textrm{ in } Q, \\
q(0,g) &= 0 &\textrm{ on } \Xi,\\
q(0;0,g) &= g & \textrm{ in } \Omega,
\end{array}
\right.
\end{equation}
and
\begin{equation}\label{Eq:decomp3}
\left\lbrace
\begin{array}{lll}
\partial_{t}q(0,0) + (-\Delta)^{s}q(0,0) &= f \textcolor{white}{+ v} &\textrm{ in } Q, \\
q(0,0) &= 0 &\textrm{ on } \Xi,\\
q(0;0,0) &= 0 & \textrm{ in } \Omega.
\end{array}
\right.
\end{equation}
Thus, according to the assumption on the data, $q(v,0)$, $q(0,g)$ and $q(0,0)$ belong to $ W(0,T)$ and we have the following result.
\begin{lemma}\label{lemma1-No-regret}
For any $v\in L^2(Q)$, $g\in L^2(\Omega)$, we have
\begin{equation}\label{Eq1:lemma1-No-regret}
\begin{array}{llll}
J_\gamma(v,g)-J_\gamma(0,g)&=&J_\gamma(v,0)-J_\gamma(0,0)\\ &+&
 \dis 2\int_Q \left[q(0,g)-q(0,0)\right] \left[q(v,0)-q(0,0)\right]dx~dt-
\gamma\|g\|^2_{L^2(\Omega)},
\end{array}
\end{equation}
where $J_\gamma$ is the functional defined in \eqref{Eq:functional}, $q(v,0)$, $q(0,g)$, $q(0,0)$ are respectively solutions of \eqref{Eq:decomp1}, \eqref{Eq:decomp2} and \eqref{Eq:decomp3}.
\end{lemma}
\begin{proof}
The proof of this lemma relies on simple computations. Indeed we first observe that one can decompose $q(v,g)$ as
\begin{equation}\label{Eq1:lemma1-proof}
q(v,g)= q(v,0) + q(0,g) - q(0,0),
\end{equation}
so that replacing $q(v,g)$ by its expression in \eqref{Eq:functional}  yields
\begin{equation}\label{Eq3:lemma1-proof}
J_\gamma(v,g) = \left \|\left(q(v,0)-z_{d}\right)+ \left(q(0,g)-z_{d}\right)-\left(q(0,0)-z_{d}\right)\right \|_{L^{2}(Q)}^{2} + \aleph \left \| v \right \|^{2}_{L^{2}(Q)}-\gamma\|g\|^2_{L^2(\Omega)}.
\end{equation}
After some computations and identifications of some of the terms such as
\[
\begin{array}{rll}
J_\gamma(v,0) &=& \dis \int_{Q} \left[q(v,0)- z_{d} \right]\left[q(v,0)- z_{d} \right] dt~dx  + \aleph \parallel v \parallel^{2}_{L^{2}(Q)}, \\
J_\gamma(0,g) &=& \dis \int_{Q} \left[q(0,g)- z_{d} \right]\left[q(v,0)- z_{d} \right] dt~dx-\gamma\|g\|^2_{L^2(\Omega)}, \\
J_\gamma(0,0) &= &\dis \int_{Q} \left[q(0,0)- z_{d} \right]\left[q(0,0)- z_{d} \right] dt~dx,
\end{array}
\]
its comes that \eqref{Eq3:lemma1-proof} takes the form
\[
\begin{array}{rll}
J_\gamma(v,g)&=&J_\gamma(v,0)+J_\gamma(0,g)-J_\gamma(0,0)\\
&+&\dis 2\int_Q \left(q(v,0)-q(0,0)\right)\left(q(0,g)-q(0,0)\right)dx~dt-
\gamma\|g\|^2_{L^2(\Omega)}.
\end{array}
\]
which is the desired result.
\end{proof}
From Lemma \ref{lemma1-No-regret}, we prove the following result.

\begin{lemma}\label{lemequivalence}
Let $\gamma>0.$ Then problem \eqref{Eq: no-regret} is equivalent to the following problem: find $u^\gamma\in L^2(Q)$ solution to the optimal control:
\begin{equation}\label{Eq5:NO-Low}
\inf_{v\in L^2(Q)}J^\gamma (v),
\end{equation}
where
\begin{equation}\label{Eq6:NO-Low}
J^\gamma (v)=J_\gamma(v,0)-J_\gamma(0,0)+\frac{1}{\gamma}\|\xi(0;v)\|^2_{L^2(\Omega)},
\end{equation}
with  $J_\gamma$ given by \eqref{Eq:functional} and $\xi(v)=\xi(x,t;v)\in W(0,T)$, solution of
\begin{equation}\label{Eq2:lemma2-NO-regret}
\left\lbrace
\begin{array}{rlllllll}
-\partial_{t}\xi(v) + (-\Delta)^{s}\xi(v) &=& q(v,0)-q(0,0)  &\textrm{in}& Q, \\
\xi(v) &=& 0 &\textrm{on}& \Xi,\\
\xi(x,T;v) &=& 0 & \textrm{in}& \Omega.
\end{array}
\right.
\end{equation}
\end{lemma}
\begin{proof} From the fact that $q(v,0)-q(0,0)\in L^2(Q)$, Proposition \ref{prop1} allows to say that there exists a unique $\xi(v) \in W(0,T)$ solution to the problem \eqref{Eq2:lemma2-NO-regret}. Moreover there exists $C(s,N,T)>0$ such that  the following estimates hold:
\begin{eqnarray}
\left \| \xi(v) \right \|_{L^{2}\left((0,T);W_{0}^{s,2}(\O)\right)} &\leq & C(s,N,T) \left \| q(v,0)- q(0,0) \right \|_{L^{2}(Q)},\label{estxiv}\\
\left \| \dis \partial_{t} \xi(v) \right \|_{L^{2}\left((0,T);W_{0}^{-s,2}(\O)\right)} &\leq& C(s,N,T) \left \| q(v,0)- q(0,0) \right \|_{L^{2}(Q)}\label{estxitv}.
\end{eqnarray}
Next we set $z(g):= z =q(0,g)-q(0,0)$. Then in view of \eqref{Eq:decomp2} and \eqref{Eq:decomp3}, $z$ verifies
\begin{equation}\label{Eq1:lemma2-proof}
\left\lbrace
\begin{array}{rlllllll}
\dis \partial_{t} z(g) + (-\Delta)^{s}z(g) &=& 0  &\textrm{in}& Q, \\
z(g) &=& 0 &\textrm{on}& \Xi,\\
z(T,v) &=&g & \textrm{in}& \Omega,
\end{array}
\right.
\end{equation}
As $g\in L^{2}(\Omega)$, we have that \eqref{Eq1:lemma2-proof} has a unique solution $z(g)\in W(0,T)$. Now, multiplying the first equation in \eqref{Eq2:lemma2-NO-regret} by $z(g)$ solution of \eqref{Eq1:lemma2-proof} and integrating by parts over $Q$, we get
\begin{equation}\label{Eq2:lemma2-proof}
\int_\Omega g\xi(0)dx=\int_Q\left(q(v,0)-q(0,0)\right)\left(q(0,g)-q(0,0)\right)dx~dt.
\end{equation}
Combining this latter identity with \eqref{Eq1:lemma1-No-regret}, we get
\begin{equation}\label{Eq3:lemma2-proof}
\dis J_\gamma(v,g)-J_\gamma(0,g)=J_\gamma(v,0)-J_\gamma(0,0)+2\int_{\Omega} g\xi(0)dx~dt-\gamma\|g\|^2_{L^2(\Omega)}.
\end{equation}
Using the Legendre-Fenchel transform, we obtain that
\[
\sup_{g\in L^2(\Omega)}\left(2\int_\Omega g\xi(0)dx-\gamma\|g\|^2_{L^2(\Omega)}\right)= \frac{1}{\gamma}\|\xi(0;v)\|^2_{L^2(\Omega)}
\]
and  problem \eqref{Eq: no-regret} is equivalent to Problem \eqref{Eq5:NO-Low}, namely: find $u^\gamma\in L^2(Q)$ solution to
\[
\inf_{v\in L^2(Q)} J^\gamma (v),
\]
where
\[
J^\gamma (v)= J_{\gamma}(v,0)- J_{\gamma}(0,0)+\frac{1}{\gamma}\|\xi(0;v)\|^2_{L^2(\Omega)}.
\]
\end{proof}
\begin{remark}\label{rem2}
Note that in the case of no-regret control problem which correspond to the case $\gamma\to 0$, the relation \eqref{Eq3:lemma2-proof} becomes \begin{equation}\label{Eq3:lemma2-proofno}
\dis J(v,g)-J(0,g)=J(v,0)-J(0,0)+2\int_Q g\xi(0)dx~dt,
\end{equation}
and the no-regret control belongs belongs to a set $\mathcal{U}$ defined by
\[
\mathcal{U} = \left\lbrace v\in L^2(Q) ~\Big|~\left(\int_{\Omega}\xi(0,x;v)g dx \right) = 0,~~\text{for all}~g \in L^{2}(\Omega) \right\rbrace.
\]
\end{remark}
Problem \eqref{Eq5:NO-Low} is a classical optimal control problem. Using minimizing sequence, we will prove that this problem has a unique solution $u^\gamma$ that we will characterize.

\begin{theorem}\label{NO-Low-Theorem1}
There exists a unique  control $u^\gamma\in L^2(Q)$ which satisfies \eqref{Eq5:NO-Low}.
\end{theorem}
\begin{proof}
Observing that
\begin{equation}\label{Eq1:NO-Low-Theorem1-proof}
J^\gamma(v)=J_\gamma(v,0)- J_\gamma(0,0)+\frac{1}{\gamma}\|\xi(0;v)\|^{2}_{L^{2}(\Omega)}\geq - J_\gamma(0,0),
\end{equation}
we can say that $\displaystyle \inf_{v\in L^{2}(Q)} J^\gamma(v)$ exists.\medskip

So let $(v_n)\in L^2(Q)$ be a minimizing sequence of $J^\gamma$. That is
\begin{equation}\label{Eq2:NO-Low-Theorem1-proof}
\dis \lim_{n\to \infty} J^\gamma(v_n)=\inf_{v\in L^{2}(Q)} J^\gamma(v).
\end{equation}
This implies that there exist a constant $C$ which does not depends on $n$ such that
\begin{equation}\label{Eq3:NO-Low-Theorem1-proof}
J^\gamma(v_n)\leq C.
\end{equation}
It then follows from the definition of $J_\gamma$ given by \eqref{Eq:functional} that
\begin{equation}\label{Eq5:NO-Low-Theorem1-proof}
\|q(v_n,0)-z_d\|^2_{L^2(Q)}+ \aleph \|v_n\|^{2}_{L^{2}(Q)}+\frac{1}{\gamma}\|\xi(0;v_n)\|^{2}_{L^{2}(\Omega)}\leq C,
\end{equation}
where  $q^n:=q(v_n,0)$ and $\xi^n=\xi(v_n,0)$ are solutions of the following equations
\begin{equation}\label{Eq6:NO-Low-Theorem1-proof}
\left\lbrace
\begin{array}{lllllll}
\dis \partial_{t} q^n + (-\Delta)^{s}q^n &=& f + v_{n} &\textrm{in}& Q, \\
q^n &=& 0 &\textrm{on}& \Xi,\\
q^n(0) &=& 0 & \textrm{in}& \Omega,
\end{array}
\right.
\end{equation}
\begin{equation}\label{Eq7:NO-Low-Theorem1-proof}
\left\lbrace
\begin{array}{llllllll}
-\dis \partial_{t} \xi^n + (-\Delta)^{s}\xi^n &= &q(v_n,0)-q(0,0) &\textrm{in}& Q, \\
\xi^n &=& 0 &\textrm{on}& \Xi,\\
\xi^n(T) &=& 0 & \textrm{in}& \Omega.
\end{array}
\right.
\end{equation}
In view of \eqref{Eq5:NO-Low-Theorem1-proof}, we have
\begin{eqnarray}
&& \|q^n\|_{L^2(Q)}\leq C, \label{Eq7:NO-Low-Theorem1-proof-a}\\
&& \|v_n\|_{L^{2}(Q)}\leq C, \label{Eq7:NO-Low-Theorem1-proof-b} \\
&& \|\xi^n(0)\|_{L^{2}(\Omega)}\leq C\sqrt{\gamma}. \label{Eq7:NO-Low-Theorem1-proof-c}
\end{eqnarray}

Since $q^n$ and $\xi^n$ are solution of nonlocal fractional diffusion equations, and from \eqref{Eq7:NO-Low-Theorem1-proof-a}, \eqref{Eq7:NO-Low-Theorem1-proof-c} and Remark \ref{rem1}, we  deduce that
\begin{eqnarray}\label{Eq8:NO-Low-Theorem1-proof}
 \|q^n\|_{L^{2}((0,T);W_{0}^{s,2}(\O))}&\leq& C, \label{Eq8:NO-Low-Theorem1-proof-a} \\
 \dis \left\|\dis \partial_{t} q^n \right\|_{L^2((0,T); W^{-s,2}(\O))}&\leq& C, \label{Eq8:NO-Low-Theorem1-proof-b} \\
\|\xi^n\|_{L^2((0,T);W_0^{s,2}(\O))}&\leq& C, \label{Eq8:NO-Low-Theorem1-proof-c}\\
\dis \left\|\dis \partial_{t} \xi^n \right\|_{L^2((0,T); W^{-s,2}(\O))}&\leq& C. \label{Eq8:NO-Low-Theorem1-proof-d}
\end{eqnarray}
We deduce that there exists $q^\gamma,\xi^\gamma\in W(0,T),$   $u^\gamma\in L^2(Q)$  and $\beta\in L^2(\Omega)$ such that
\begin{eqnarray}
 q^n&\rightharpoonup& q^\gamma \hbox{ weakly in }L^2((0,T);W_0^{s,2}(\O)),\label{Eq9:NO-Low-Theorem1-proof11} \\
 \dis \partial_{t} q^n &\rightharpoonup&\dis \partial_{t} q^\gamma \hbox{ weakly in }L^2((0,T);W^{-s,2}(\O)),\label{Eq9:NO-Low-Theorem1-proof11b} \\
 \xi^n&\rightharpoonup& \xi^\gamma \hbox{ weakly in }L^2((0,T);W_0^{s,2}(\O)),\label{Eq9:NO-Low-Theorem1-proof12} \\
 \dis \partial_{t} \xi^n &\rightharpoonup&\dis \partial_{t} \xi^\gamma\hbox{ weakly in }L^2((0,T);W^{-s,2}(\O)),\label{Eq9:NO-Low-Theorem1-proof12b} \\
 v_n&\rightharpoonup& u^\gamma \hbox{ weakly in }L^2(Q), \label{Eq9:NO-Low-Theorem1-proof14}\\
 \xi^n(0)&\rightharpoonup& \beta \hbox{ weakly in } L^2(\Omega).\label{Eq9:NO-Low-Theorem1-proof13}
\end{eqnarray}
\begin{remark} In view of \eqref{Eq9:NO-Low-Theorem1-proof11} and \eqref{Eq9:NO-Low-Theorem1-proof12}
we have on one hand that
 \begin{eqnarray}
 q^\gamma &=~0& \hbox{ on } \Xi, \label{Eq9:NO-Low-Theorem1-proof11a0} \\
 \xi^\gamma&=0& \hbox{ on } \Xi,\label{Eq9:NO-Low-Theorem1-proof12a0}
 \end{eqnarray}
 because $L^2((0,T);W_0^{s,2}(\O))$ is a Hilbert space, and on the other hand,
 \begin{eqnarray}
 q^n&\rightharpoonup& q^\gamma \hbox{ weakly in }L^2((0,T);W_0^{s,2}(\R^N)),\label{Eq9:NO-Low-Theorem1-proof11a} \\
 \xi^n&\rightharpoonup& \xi^\gamma \hbox{ weakly in }L^2((0,T);W_0^{s,2}(\R^N)),\label{Eq9:NO-Low-Theorem1-proof12a} .
 \end{eqnarray}
 because  $q^n=\xi^n =0 $ in $\Xi$.
\end{remark}
The rest of the proof will be subdivide in two steps.\medskip

{\sc Step 1.} We we first prove that $q^\gamma= q (u^\gamma,0)$  and $\xi^\gamma =\xi (u^\gamma)$ satisfy respectively \eqref{Eq:decomp1} and \eqref{Eq2:lemma2-NO-regret}.  \par

We set 
\[
\mathbb{D}\left(\Omega\times (0,T)\right)=\left\lbrace \phi\in C_0^\infty (\R^N\times (0,T)) \hbox{ such that } \phi=0~\text{in} ~\R^{N}\setminus\Omega \right\rbrace.
\]
If we multiplying  the first equation in \eqref{Eq6:NO-Low-Theorem1-proof} by
$\phi \in \mathbb{D}\left(\Omega\times (0,T)\right)$ and integrate by parts, we obtain  that
\[
\begin{array}{llllll}
\dis \int_{0}^T\int_\Omega \partial_{t} q^n ~\phi dx~dt + \dis \frac{C_{N,s}}{2}\int_0^T\int_{\R^N}
\int_{\R^N}\frac{(q^n(x,t)-q^n(y,t))(\phi(x,t)-\phi(y,t))}{|x-y|^{N+2s}}\,dx~dy~dt =\\ 
\dis \int_{0}^T \int_\Omega (f+v_n)\phi dx~dt, \quad \forall \phi\in \mathbb{D}\left(\Omega\times (0,T)\right).
\end{array}
\] 
Passing this latter identity to the limit when $n\to 0$ while using \eqref{Eq9:NO-Low-Theorem1-proof11}-\eqref{Eq9:NO-Low-Theorem1-proof14} and \eqref{Eq9:NO-Low-Theorem1-proof11a}-\eqref{Eq9:NO-Low-Theorem1-proof12a}, we have that
\[
\begin{array}{llllll}
\dis \int_{0}^T\int_\Omega ~\phi \partial_{t} q^\gamma  dx~dt + \dis \frac{C_{N,s}}{2}\int_0^T\int_{\R^N}
\int_{\R^N}\frac{(q^\gamma(x,t)-q^\gamma(y,t))(\phi(x,t)-\phi(y,t))}{|x-y|^{N+2s}}\,dx~dy~dt =\\
\dis  \int_{0}^T \int_\Omega \phi (f+u^\gamma) dx~dt,\quad \forall \phi\in \mathbb{D}\left(\Omega\times (0,T)\right).
\end{array}
\]
Therefore,  using the formula of integration by parts given by \eqref{Eq:Int-Part}, we deduce that
\[
\begin{array}{llllll}
\dis \int_{0}^T\int_\Omega \phi \partial_{t} q^\gamma~ dx~dt + \dis \int_0^T\int_{\Omega} \phi(-\Delta)^s q^\gamma dx~dt =
\dis  \int_{0}^T \int_\Omega \phi(f+u^\gamma) dx~dt,\quad \forall \phi \in \mathbb{D}\left(\Omega\times (0,T)\right).
\end{array}
\]
Hence we deduce that
\begin{equation}\label{inter1}
\partial_{t} q^{\gamma} + (-\Delta)^s q^{\gamma} = f+u^\gamma \hbox{ in } Q.\end{equation}

Now, multiplying  the first equation in \eqref{Eq6:NO-Low-Theorem1-proof} by $\phi\in C^\infty (\R^N\times (0,T))$  with $\phi=0$ on $\Xi$ and $\phi(T)=0$ in $\Omega$, we have that
\begin{equation}\label{inter2}
\begin{array}{llllll}
\dis -\int_{0}^{T}\int_{\Omega} \partial_{t} \phi ~q^n~ dx~dt + \dis \frac{C_{N,s}}{2}\int_0^T\int_{\R^N}
\int_{\R^N}\frac{(q^n(x,t)-q^n(y,t))(\phi(x,t)-\phi(y,t))}{|x-y|^{N+2s}}\,dx~dy~dt =\\
\dis  \int_{0}^T \int_\Omega (f+v_n)\phi ~dx\,dt, \quad \forall \phi~\in \mathbb{D}\left(\Omega\times (0,T)\right).
	\end{array}
\end{equation}
Passing \eqref{inter2} to the limit when $n\to 0$ while using \eqref{Eq9:NO-Low-Theorem1-proof11}-\eqref{Eq9:NO-Low-Theorem1-proof14} and \eqref{Eq9:NO-Low-Theorem1-proof11a}-\eqref{Eq9:NO-Low-Theorem1-proof12a}, we get
\begin{equation}\label{inter3}
\begin{array}{llllll}
\dis -\int_{0}^T\int_\Omega \partial_{t} \phi~q^\gamma ~dx~dt + \dis \frac{C_{N,s}}{2}\int_0^T\int_{\R^N}
\int_{\R^N}\frac{(q^\gamma(x,t)-q^\gamma(y,t))(\phi(x,t)-\phi(y,t))}{|x-y|^{N+2s}}\,dx~dy~dt =\\
\dis  \int_{0}^T \int_\Omega (f+u^\gamma)\phi~ dx\,dt,\\
 \forall \phi\in C^\infty (\R^N\times (0,T)) \hbox{ such that } \phi=0 \hbox{ on } \Xi \hbox{ and }\phi(T)=0 \hbox{ in }\Omega.
	\end{array}
\end{equation}
In view of \eqref{inter1}, $(-\Delta)^{s} q^\gamma \in L^2(Q)$ and consequently in $L^2(\R^N)$ since $q^\gamma=0$ in $\Xi$. On the other hand, as $q^\gamma \in W(0,T)$, we have that $q^\gamma(0)$ and $q^\gamma(T)$ exist and belong to $L^2(\Omega)$. Therefore, using \eqref{inter1} in \eqref{inter3}, we get
\[
\begin{array}{llllll}
\dis \int_\Omega q^{\gamma} (0)\phi(0) dx +\int_{0}^T\int_\Omega \partial_{t} q^{\gamma}~\phi dx~dt + \dis \int_0^T\int_{\Omega} (-\Delta)^{s}q^{\gamma} \phi dx~dt =\\
\dis  \int_{0}^T \int_\Omega (f+u^\gamma)\phi dx\,dt,
\\
\forall \phi\in C^\infty (\R^N\times (0,T)) \hbox{ such that } \phi=0 \hbox{ on } \Xi \hbox{ and }\phi(T)=0 \hbox{ in }\Omega,
\end{array}
\]
 which in view of \eqref{inter1} yields
\[
\begin{array}{llllll}
\dis \int_\Omega q^\gamma (0)\phi(0) dx =0,
\\
\forall \phi\in C^\infty (\R^N\times (0,T)) \hbox{ such that } \phi=0 \hbox{ on } \Xi \hbox{ and }\phi(T)=0 \hbox{ in }\Omega.
	\end{array}
\]
Hence
\begin{equation}\label{inter4}
q^\gamma(0)=0 \hbox{ in } \Omega.
\end{equation}
It then follows from \eqref{Eq9:NO-Low-Theorem1-proof11a0}, \eqref{inter1} and \eqref{inter4} that $q^\gamma= q(u^\gamma)$ is solution of

\begin{equation}\label{Eq16:NO-Low-Theorem1-proof}
\left\lbrace
\begin{array}{llllllll}
\dis \partial_{t} q^{\gamma}(u^\gamma,0) + (-\Delta)^{s}q^{\gamma}(u^\gamma,0) &=& f+u^\gamma &\textrm{in}& Q, \\
q^{\gamma}(u^\gamma,0) &=& 0 &\textrm{on}& \Xi,\\
q^{\gamma}(0,u^\gamma,0) &=&0 & \textrm{in}& \Omega.
\end{array}
\right.
\end{equation}
Proceeding as above, we prove that   $\xi^\gamma=\xi(u^\gamma)$ satisfies \begin{equation}\label{Eq:nolowregret2}
\left\lbrace
\begin{array}{llllll}
-\dis \partial_{t}\xi^\gamma(u^\gamma) + (-\Delta)^{s}\xi^\gamma(u^\gamma) &=& q^\gamma(u^\gamma,0)-q(0,0) &\textrm{in}& Q,  \\
\xi^\gamma(u^\gamma) &= &0 &\textrm{on}& \Xi,\\
\xi^\gamma(T,u^\gamma) &=& 0 & \textrm{in}& \Omega.
\end{array}
\right.
\end{equation}
Moreover using \eqref{Eq9:NO-Low-Theorem1-proof12b} and \eqref{Eq9:NO-Low-Theorem1-proof13}, we have that
\begin{equation}\label{Eq21:NO-Low-Theorem1-proof}
\xi^\gamma(0)= \xi(0;u^\gamma)=\beta.
\end{equation}

\noindent{\sc Step 2.} From weak lower semi-continuity of the function $v\mapsto J^\gamma(v)$, we obtain
\[
J^\gamma(u^\gamma)\leq \lim_{n\rightarrow \infty}\inf J^\gamma(v_n).
\]
Hence, according to \eqref{Eq2:NO-Low-Theorem1-proof}, we deduce that
\[
J^\gamma(u^\gamma)= \inf_{v\in L^2(Q)} J^\gamma(v).
\]
The uniqueness of $u^\gamma$ is straightforward from the strict convexity of $J^\gamma$.
\end{proof}

\begin{theorem}\label{NO-Low-Theorem2}
Let $u^\gamma\in L^2(Q)$ be the optimal control solution of \eqref{Eq5:NO-Low}. Then there exist $\psi^\gamma\in W(0,T)$ and $\phi^\gamma\in W(0,T)$ such that $(q^\gamma := q(u^\gamma,0), \xi^\gamma, \psi^\gamma, \phi^\gamma)$ satisfies the following optimality system:

\begin{equation}\label{Eq:qgama}
\left\lbrace
\begin{array}{lllllll}
\dis \partial_{t} q^\gamma + (-\Delta)^{s}q^\gamma &=& f + u^\gamma&\hbox{in}& Q,\\
q^\gamma &=& 0 &\hbox{on}&  \Xi,\\
q^\gamma(0) &=& 0 &\hbox{in}&  \Omega,
\end{array}
\right.
\end{equation}
\begin{equation}\label{xig}
\left\lbrace
\begin{array}{lllllll}
-\dis \partial_{t} \xi^\gamma  + (-\Delta)^{s}\xi^\gamma &=& q(u^\gamma,0)-q(0,0) &\hbox{in}&  Q, \\
\xi^\gamma &=& 0 &\hbox{on}&  \Xi,\\
\xi^\gamma(T) &= &0 &\hbox{in}& \Omega,
\end{array}
\right.
\end{equation}
\begin{equation}\label{Eq:spigama}
\left\lbrace
\begin{array}{lllllll}
\dis \partial_{t} \psi^\gamma + (-\Delta)^{s}\psi^\gamma &=& 0 &\hbox{in}&  Q,\\
\psi^\gamma &=& 0 \textcolor{white}{q^\gamma(u^\gamma,0)-z_d} &\hbox{on}&  \Xi,\\
\psi^\gamma(0) &= &-\frac{1}{\sqrt{\gamma}}\xi^\gamma(0) &\hbox{in}& \Omega,
\end{array}
\right.
\end{equation}
\begin{equation}\label{Eq:phigama}
\left\lbrace
\begin{array}{llllllll}
-\dis \partial_{t} \phi^\gamma + (-\Delta)^{s}\phi^\gamma &= & q^\gamma +z_d-\frac{1}{\sqrt{\gamma}}\psi^\gamma  &\hbox{in}&  Q, \\
\phi^\gamma &=& 0 &\hbox{on}&  \Xi,\\
\phi^\gamma(T) &=& 0 &\hbox{in}&  \Omega,
\end{array}
\right.
\end{equation}
and
\begin{equation}\label{Eq:ugama}
\aleph u^\gamma+\phi^\gamma=0 \hbox { in } Q.
\end{equation}
\end{theorem}

\begin{proof}
Relations \eqref{Eq16:NO-Low-Theorem1-proof} give \eqref{Eq:qgama}. To prove \eqref{Eq:spigama}, \eqref{Eq:phigama} and \eqref{Eq:ugama}, we express the Euler-Lagrange optimality conditions, which characterize the low-regret control $u^\gamma$:
\begin{equation}\label{Eq:theorem2-proof}
\frac{d}{d\tau}J^\gamma(u^\gamma+\tau(v-u^\gamma)\Big|_{\tau=0}=0,~~~\text{for all }~~v\in L^2(Q).
\end{equation}
Let $z^\gamma:=q(u^\gamma+\tau(v-u^\gamma),0)-q(u^\gamma,0)$ be the state associated with the control $(v-u^\gamma)$. Then $z$ is solution of

\begin{equation}\label{Eq2:theorem2-proof}
\left\lbrace
\begin{array}{llllllll}
\dis \partial_{t} z^\gamma + (-\Delta)^{s}z^\gamma &=& (v-u^\gamma) &\textrm{in}& Q, \\
z^\gamma &=& 0 &\textrm{on}& \Xi,\\
z^\gamma(0) &=& 0 & \textrm{in}& \Omega,
\end{array}
\right.
\end{equation}
and after computations, \eqref{Eq:theorem2-proof} gives
\begin{equation}\label{Eq3:theorem2-proof}
\begin{array}{lllll}
\dis \int_Q z^\gamma(q^\gamma-z_d)dx~dt+\int_{Q}\aleph u^\gamma(v-u^\gamma)dx~dt+\frac{1}{\gamma}\int_\Omega \xi(0;v-u^\gamma)\xi(0; u^\gamma) dx = 0, ~~\forall ~v\in L^2(Q),
\end{array}
\end{equation}
where $\xi(v-u^\gamma)$ is a solution of
\begin{equation}\label{Eq4:theorem2-proof}
\left\lbrace
\begin{array}{lllllll}
-\dis \partial_{t}\xi(v-u^\gamma) + (-\Delta)^{s}\xi(v-u^\gamma) &=& z(v-u^\gamma)  &\textrm{in}& Q, \\
\xi(v-u^\gamma) &=&0 &\textrm{on}& \Xi,\\
\xi(T; v-u^\gamma) &=& 0 & \textrm{in}& \Omega.
\end{array}
\right.
\end{equation}
To give an interpretation of \eqref{Eq3:theorem2-proof} we use the adjoint states $\phi^\gamma$ and $\psi^\gamma$  solution  of \eqref{Eq2:theorem2-proof} and \eqref{Eq3:theorem2-proof}. So if we multiply the first equations in \eqref{Eq2:theorem2-proof} and \eqref{Eq4:theorem2-proof} by $\phi^\gamma$ and $\frac{1}{\sqrt{\gamma}}\psi^\gamma$  and integrate  by part over $Q$, we get

\begin{equation}\label{Eq5:theorem2-proof}
\int_Q z^\gamma\left(q^\gamma + z_d-\frac{1}{\sqrt{\gamma}}\psi^\gamma\right)dx~dt = \int_{Q}(v-u^\gamma)\phi^\gamma dx~dt
\end{equation}
and
\begin{equation}\label{Eq6:theorem2-proof}
\frac{1}{{\gamma}}\int_\Omega\xi^\gamma(0)\xi^\gamma(0;v-u^\gamma)~dx= \frac{1}{\sqrt{\gamma}}\int_{Q}z^\gamma\psi^\gamma dx~dt.
\end{equation}
Therefore, combining  \eqref{Eq3:theorem2-proof} with  \eqref{Eq5:theorem2-proof} and \eqref{Eq6:theorem2-proof}, we obtain
\[
\aleph u^\gamma+\phi^\gamma=0 \qquad \text{in}\qquad Q.
\]
\end{proof}
We will now study the convergence of the sequences $u^\gamma, q^{\gamma}, \xi^\gamma, \phi^\gamma$ and $\psi^\gamma$.

\begin{proposition}
The optimal control $u^\gamma$ converges in $L^2(Q)$ to $u$ solution of  \eqref{Eq: no-regretintro}. Moreover, there exists $q=q(u,0),$
 $\tau_1$ $\tau_2$, $\psi$ and $\phi$ such that $(u, q := q(u,0), \xi, \psi,\phi)$ satisfies the following optimality system:
\begin{equation}\label{eqqf}
\left\lbrace
\begin{array}{lllllll}
\dis \partial_{t} q  + (-\Delta)^{s}q  &=& f + u \textcolor{white}{q(0,0)++}&\hbox{in}& Q,\\
q &=& 0 &\hbox{on}&  \Xi,\\
q(0) &=& 0 &\hbox{in}&  \Omega,
\end{array}
\right.
\end{equation}
\begin{equation}\label{eqxif}
\left\lbrace
\begin{array}{lllllll}
-\dis \partial_{t} \xi + (-\Delta)^{s}\xi &=& q(u,0)-q(0,0) &\hbox{in}&  Q, \\
\xi &=& 0 &\hbox{on}&  \Xi,\\
\xi(T) &= &0 &\hbox{in}& \Omega,
\end{array}
\right.
\end{equation}
\begin{equation}\label{eqpsif}
\left\lbrace
\begin{array}{lllllll}
\dis \partial_{t} \psi + (-\Delta)^{s}\psi &=& 0 &\hbox{in}&  Q, \\
\psi &=& 0  &\hbox{on}&  \Xi,\\
\psi^\gamma(0) &= &\tau_1 &\hbox{in}& \Omega,
\end{array}
\right.
\end{equation}
\begin{equation}\label{eqphi}
\left\lbrace
\begin{array}{llllllll}
-\dis \partial_{t} \phi  + (-\Delta)^{s}\phi &= & q +z_d-\tau_2  &\hbox{in}&  Q, \\
\phi &=& 0 &\hbox{on}&  \Xi,\\
\phi(T) &=& 0 &\hbox{in}&  \Omega,
\end{array}
\right.
\end{equation}
and
\begin{equation}\label{equ}
\aleph u+\phi=0 \hbox { in } Q.
\end{equation}
\end{proposition}

\begin{proof} In view of the definition of the functional $J_{\gamma}$ and Remark \ref{cont_unique}, we have that $J_{\gamma}(0) = 0$. Hence,
\begin{equation}\label{Eq:Jgama}
J_{\gamma}(u^{\gamma})\leq J_{\gamma}(0) = 0,
\end{equation}
because for any $\gamma>0$, $u^{\gamma}$ is  solution of the  problem \eqref{Eq5:NO-Low}.
It then follows from \eqref{Eq:Jgama} and the  definition of the functional $J_{\gamma}$ that	
\[
\left\|q^{\gamma}-z_{d}\right\|^{2}_{L^{2}(Q)} + \aleph \left\|u^{\gamma}\right\|^{2}_{L^{2}(Q)} +\frac{1}{\gamma}\left\|\xi^\gamma (0)\right\|^2_{L^2(Q)}\leq J(0,0)
\]
As a consequence, there exists a positive constant $C$, independent of $\gamma$ such that
\begin{eqnarray}
&& \|q^{\gamma}\|_{L^2(Q)}\leq C, \label{Eq:estimateatinitial-1}\\
&& \|u^{\gamma}\|_{L^{2}(Q)}\leq C, \label{Eq:estimateatinitial-2} \\
&& \|\xi^{\gamma}(0)\|_{L^{2}(\Omega)}\leq C\sqrt{\gamma},\label{Eq:estimateatinitial-3}\\
&& \frac{1}{\sqrt{\gamma}}\|\xi^{\gamma}(0)\|_{L^{2}(\Omega)}\leq C.\label{Eq:estimateatinitial-4}
\end{eqnarray}

since $q^\gamma,$  $\xi^\gamma$  and $\psi^\gamma$ are respectively solutions of nonlocal fractional diffusion equations \eqref{Eq:qgama}, \eqref{xig} and \eqref{Eq:spigama}, from \eqref{Eq:estimateatinitial-2}, \eqref{Eq:estimateatinitial-1},\eqref{Eq:estimateatinitial-4} and Remark \ref{rem1}, we  deduce that
\begin{eqnarray}
 \|q^\gamma\|_{L^{2}((0,T);W_{0}^{s,2}(\O))}&\leq& C, \label{Eq8:NO-Low-Theorem1-proof-af} \\
 \dis \left\|\dis \partial_{t} q^\gamma \right\|_{L^2((0,T); W^{-s,2}(\O))}&\leq& C, \label{Eq8:NO-Low-Theorem1-proof-bf} \\
\|\xi^\gamma\|_{L^2((0,T);W_0^{s,2}(\O))}&\leq& C, \label{Eq8:NO-Low-Theorem1-proof-cf}\\
\dis \left\|\dis \partial_{t} \xi^\gamma \right\|_{L^2((0,T); W^{-s,2}(\O))}&\leq& C, \label{Eq8:NO-Low-Theorem1-proof-df}\\
\|\psi^\gamma\|_{L^{2}((0,T);W_{0}^{s,2}(\O))}&\leq& C, \label{estpsig} \\
 \dis \left\|\dis \partial_{t} \psi^\gamma \right\|_{L^2((0,T); W^{-s,2}(\O))}&\leq& C. \label{estpsitg}
\end{eqnarray}
We deduce that there exists $q,\xi,\psi\in W(0,T)$, $u^\gamma\in L^2(Q)$ and $\beta,\tau_1\in L^2(\Omega)$ such that
\begin{eqnarray}
 q^\gamma &\rightharpoonup & q \hbox{ weakly in }L^2((0,T);W_0^{s,2}(\O)),\label{convqgamma} \\
 \dis \partial_{t} q^\gamma &\rightharpoonup&\dis \partial_{t} q \hbox{ weakly in }L^2((0,T);W^{-s,2}(\O)),\label{convqtgamma} \\
 \xi^\gamma&\rightharpoonup& \xi \hbox{ weakly in }L^2((0,T);W_0^{s,2}(\O)),\label{convxigamma} \\
 \dis \partial_{t} \xi^\gamma &\rightharpoonup&\dis \partial_{t} \xi \hbox{ weakly in }L^2((0,T);W^{-s,2}(\O)),\label{convxitgamma} \\
 u^\gamma&\rightharpoonup& u \hbox{ weakly in }L^2(Q), \label{convugamma}\\
 \xi^\gamma(0)&\rightharpoonup& \beta \hbox{ weakly in } L^2(\Omega),\label{convxi0gamma}\\
 \dis \frac{1}{\sqrt{\gamma}}\xi^\gamma(0)&\rightharpoonup& \tau_1 \hbox{ weakly in } L^2(\Omega).\label{convrxi0gamma}
\end{eqnarray}
Consequently, we prove as in pages \pageref{Eq9:NO-Low-Theorem1-proof11a0}--\pageref{Eq16:NO-Low-Theorem1-proof} that $q=q(u,0)\in W(0,T),$  $\xi=\xi(u)\in W(0,T)$ and $\psi\in W(0,T)$  satisfy respectively \eqref{eqqf}, \eqref{eqxif} and \eqref{eqpsif}.\medskip

Using \eqref{Eq:ugama} and \eqref{Eq:estimateatinitial-2}, we have that there exists $C>0$ independent of $\gamma$
\[
\|\phi^\gamma\|_{L^2(Q)}\leq C.
\]
Hence there exists $\phi\in L^2(Q)$ such that
\begin{equation}\label{convphiL2}
\phi^\gamma \rightharpoonup \phi \hbox{ weakly in } L^2(Q).
\end{equation}
Next, if we multiply the first equation in \eqref{Eq:phigama} by $\rho \in \mathbb{D}\left(\Omega\times (0,T)\right)$ and integrate by parts, we get
\[
\dis \int_Q \phi^\gamma \left(-\partial_{t} \rho + (-\Delta)^s\rho\right)~dx~dt -\int_Q (q^\gamma -z_d)\rho~dx~dt =
\dis\int_Q\frac{1}{\sqrt{\gamma}} \psi^\gamma \rho~ dx~dt.
\]
Passing this latter inequality  at the limit when $\gamma \to 0$ while using \eqref{convphiL2} and \eqref{Eq:estimateatinitial-1}, we obtain that
\[
\begin{array}{llll}
\dis \int_Q \phi \left(-\partial_{t} \rho + (-\Delta)^s\rho\right)~dx~dt -\int_Q (q -z_d)\rho~ dx~dt &=&
\dis\lim_{\gamma\to 0}\int_Q\frac{1}{\sqrt{\gamma}} \psi^\gamma \rho~ dx~ dt,\\
\forall \rho \in \mathbb{D}\left(\Omega\times (0,T)\right),
\end{array}
\]
which  by integrating by parts again gives
\[
\begin{array}{llll}
\dis \int_Q \rho \left(\partial_{t} \phi + (-\Delta)^s\phi\right)~dx~dt -\int_Q (q -z_d)\rho~dx~dt &=&\dis\lim_{ \gamma\to 0}\int_Q\frac{1}{\sqrt{\gamma}} \psi^\gamma \rho~dx~dt,\\
\forall \rho \in \mathbb{D}\left(\Omega\times (0,T)\right).
\end{array}
\]
Hence we deduce  that there there exists $\tau_2$ such that
\[
\partial_{t} \phi +(-\Delta)^s\phi = q -z_d+\tau_2 \in Q,
\]
and it follows from  second and third equations in \eqref{Eq:phigama} that $\phi$ satisfies \eqref{eqphi}.
Now,  in view of \eqref{Eq:estimateatinitial-3} we have that
\[
\xi^\gamma (0) \rightarrow \xi(0;u) = 0  \text{ strongly in }  L^{2}(\Omega).
\]
As a consequence, we have that $\displaystyle \int_{\Omega}g\xi(x,0;u)~dx = 0,$ from which according to Remark \ref{rem2} means that  $u$ is solution of the no-regret control problem \eqref{Eq: no-regretintro}.
\end{proof}

\section{Conclusion}
The optimal control of diffusion equation with missing data is an important result in the theory of control of PDEs, and its fractional counterpart should have no less signiﬁcance in the theory of control of fractional PDEs. We discussed about the optimal control of diffusion equation with missing data governed by Dirichlet fractional Laplacian. We proved that if the control acts on $Q$ then we can characterise the limit of the optimal control problem \eqref{Eq: no-regret} by optimality system given by \eqref{eqqf}-\eqref{equ}.\medskip

\textbf{Acknowledgment:} The first and second author is grateful for the facilities provided by the German research Chairs. The third author was supported by the Alexander von Humboldt foundation, under the programme financed by the BMBF entitled ``German research Chairs".\medskip

\bibliographystyle{apa}


\end{document}